\documentclass[11pt,letterpaper]{amsart}
\usepackage{latexsym,amsfonts,amssymb,amsmath,amsthm, datetime2}
\usepackage{graphicx}
\usepackage{color}
\usepackage{mathrsfs}
\usepackage{amsmath, amsthm, amsfonts, amssymb}
\usepackage{wrapfig}
\usepackage{stackrel}
\usepackage{color}
\usepackage{float}
\usepackage{enumitem}
\usepackage{mathtools}
\usepackage{multicol}
\usepackage[normalem]{ulem}
\usepackage{xcolor}
\usepackage{hyperref}
\usepackage{cancel}

\usepackage[margin=1in]{geometry}

\newcommand{\sumstar}{\sideset{}{^*}\sum}

\newcommand{\lc}{\left(}
\newcommand{\rc}{\right)}

\newcommand{\inftyInt}{\int_{-\infty}^{\infty}}
\newcommand{\h}{\frac{1}{2}}
\newcommand{\infS}{\sum_{n,m\geq 1}}
\newcommand{\fST}{\sum_{n,m \leq T^{50}}}
\newcommand{\lv}{\left|}
\newcommand{\rv}{\right|}

\newcommand{\GT}{G_{T,H}}

\newtheorem{cor}{Corollary}
\newtheorem{thm}{Theorem}
\newtheorem{lem}{Lemma}

\theoremstyle{plain}		
	\newtheorem{mytheo}{Theorem}[section]

	\newtheorem{myremark}[mytheo]{Remark}

\theoremstyle{remark}

\numberwithin{equation}{section}

\usepackage[backend=biber, doi=false,isbn=false,url=false, articlein=false, maxnames=4, maxalphanames=4, sorting=nty, style=ext-alphabetic]{biblatex}
\DeclareFieldFormat[article,misc,inbook,incollection]{title}{#1}
\begin{document}
\author{Ze Sen Tang}
\email{zesen.tang@utdallas.edu}

 \address{Mathematical Sciences Dept. \\
 	  The University of Texas at Dallas \\
 	  Richardson\\
 	  TX 75080-3021}

\title{Reciprocity for The Short Twisted Second Moment of The Riemann Zeta Function}

\begin{abstract}
    We prove a reciprocity formula for a twisted second moment of the Riemann zeta function over a short interval of length $T^{\delta}$, where $\delta \in (\h,1)$, centered at height $T$ on the critical line. This result extends Khan's reciprocity formula to the short interval setting and makes explicit the relation between the twisting parameters and the interval length.
\end{abstract}

\maketitle

\section{Introduction}
The Riemann zeta function occupies a central position in analytic number theory, and many questions arise from the study of its behavior on the critical line Re$(s) = \h$. For instance, the famous Riemann Hypothesis predicts that all non-trivial zeros of the zeta-function lie on this line. Beginning with Selberg \cite{Sel42}, one major direction toward this problem is to prove that a positive proportion of the zeros lie on the critical line. Subsequently, twisted moments of the zeta-function played an important role in this direction through Levinson \cite{Lev74}, who proved that more than one third of the zeros lie on the line. Following Levinson's method, Conrey \cite{Con89} improved this proportion to more than two fifths, Bui, Conrey and Young \cite{BCY11} obtained more than $41.05\%$, Feng \cite{Feng12} improved to $41.28\%$, and Pratt, Robles, Zaharescu and Zeindler \cite{PRZZ20} established the record of more than five-twelfths. Very recently, a preprint authored by Claude \cite{Claude26}, an artificial-intelligence developed by Anthropic, claimed an improvement of this record to $67.25\%$. Moreover, the study of moments is closely connected with subconvexity problems. For example, alternative proofs of the classical Hardy–Littlewood–Weyl bound, namely 
\begin{equation}
    |\zeta\lc \tfrac{1}{2}+it \rc| \ll |t|^{{\frac{1}{6}+\epsilon}},
\end{equation}
can be established through moments. Atkinson \cite{Atk49} and Jutila \cite{Jut83} gave a proof of this bound from the short second moment of length $T^{\frac{1}{3}}$. Heath-Brown \cite{Hea78} obtained the bound through the twelfth moment, while Iwaniec \cite{Iwa80} obtained it from the short fourth moment of length $T^{\frac{2}{3}}$. Jutila \cite{Jut90} later gave a new proof of the same short fourth moment result.

The results discussed above motivated a deeper study of the structure of moments. In Conrey's unpublished work \cite{Con07}, he observed a reciprocity structure in the twisted second moment of Dirichlet $L$-functions. For $p$ and $q$ distinct primes,
\begin{equation}
    \frac{1}{q^{\h}}\sumstar_{\chi \bmod q} \chi(p) \lv L\lc \tfrac{1}{2}+it,\chi \rc \rv^2 = \text{main } + \frac{1}{p^{\h}}\sumstar_{\chi \bmod p} \chi(-q) \lv L\lc \tfrac{1}{2}+it,\chi \rc \rv^2 + \text{ error}.
\end{equation}
This formula exhibits a reciprocity relation that links the twisted second moment of Dirichlet $L$-function to a dual moment with the roles of parameter $p$ and $q$ interchanged. This result naturally motivates the study of reciprocity structure in moments of the zeta function. In particular, Khan \cite{khan25} proved a reciprocity formula for the twisted second moment of the zeta function. Roughly speaking, his result takes the form
\begin{equation}
    \int_{-T}^{T} \lc \frac{p}{q} \rc^{it} \lv \zeta\lc\tfrac{1}{2}+it\rc \rv^2 dt = \text{main} + \lc\frac{T}{2\pi}\rc^{\h}\frac{p^{\h}}{p-1}\sumstar_{\substack{\chi \bmod p \\ \chi(-1)=1}} \chi(q) \int_{-1}^{1} \lc \frac{T}{2q} \rc^{it} \lv L\lc\tfrac{1}{2}+it,\chi\rc \rv^2 dt + \text{error},
\end{equation}
where the main term has size $\frac{T}{\sqrt{pq}}\log\lc \frac{T}{pq} \rc$ and the error term $O\lc (pqT)^{\epsilon}\left[\lc\frac{p}{q}\rc^{\h}+\lc\frac{q}{p}\rc^{\h}\right] \rc$. By informally setting $q=1$, we observe that Khan's result yields a reciprocity formula structurally analogous to Conrey's result, where the twist parameter $p$ becomes the modulus in the dual moment.

The purpose of this paper is to establish a generalization of Khan's reciprocity formula in the short interval setting. More precisely, we study the short twisted second moment on an interval of length $H$ around a height $T$ where $H<T$. Before stating the main result, define
\begin{equation} \label{GTH}
    G_{T,H}(w) = \int_{0}^{\infty} \exp\lc -\frac{(2\pi x - T)^2}{H^2} \rc x^{w-1} dx.
\end{equation}
We prove the following theorem
\begin{thm}\label{thm1}
    Let $p,q$ be distinct odd primes, $T>1$ and $H=T^{\delta}$ for $\delta \in \lc \h,1 \rc$. We have
    \begin{equation}\label{eqthm1}
    \begin{aligned}
        \inftyInt \lc \frac{p}{q} \rc^{it} &\lv \zeta\lc \tfrac{1}{2}+it \rc \rv^2 \exp\lc -\frac{(t-T)^2}{H^2} \rc dt =\\
        &\frac{2 \pi}{\sqrt{pq}}\lc G_{T,H}'(1)-G_{T,H}(1)\log(pq)+2\gamma G_{T,H}(1) \rc \\
        &+ \sumstar_{\chi \bmod p} \frac{\sqrt{p}}{i^{\alpha}(p-1)}\chi(q) \inftyInt G_{T,H}\lc \tfrac{1}{2}+it \rc \lc \frac{\pi}{q}\rc^{it} \frac{\Gamma \lc \frac{1+2\alpha-2it}{4} \rc}{\Gamma \lc \frac{1+2\alpha+2it}{4} \rc} \lv L\lc \tfrac{1}{2}+it,\chi \rc \rv^2 dt\\
        &+ O\lc \frac{T}{H}(pqT)^{\epsilon} \left[\lc \frac{p}{q} \rc^{\h} + \lc \frac{q}{p} \rc^{\h} \right] \rc,
    \end{aligned}
    \end{equation}
    for $\alpha= 0$ if $\chi(-1)=1$ and $\alpha=1$ if $\chi(-1)=-1$.
\end{thm}
By Lemma \ref{mellin2}, the main term of \eqref{eqthm1} has size 
\begin{equation}
    \frac{1}{\sqrt{pq}}\lc G_{T,H}'(1)-G_{T,H}(1)\log(pq)+2\gamma G_{T,H}(1)\rc \asymp \frac{H}{\sqrt{pq}}\log\lc \frac{T}{pq} \rc.
\end{equation}
Moreover, as shown by Lemma \ref{weight}, $G_{T,H}\lc \tfrac{1}{2}+it \rc$ decays rapidly for $\lv t \rv \gg \frac{T}{H}$. Thus, the integral in the dual moment is essentially supported on an interval of length $\frac{T}{H}$. Two types of reciprocity appear in our result. The first is a type of arithmetic reciprocity which was already presented in Khan's result. The second is a type of analytic reciprocity in which the original length of support $H$ is transformed into $\frac{T}{H}$ in the dual moment.

Another objective of this paper is to observe the relation between the size of twisting parameter and the length of interval of integral. In general, the twisting parameter cannot be arbitrarily large, while the integral cannot be supported on an interval that is arbitrarily small. However, the relation between these two quantities has not been made explicit in previous results. In our result, this relation is given by 
\begin{equation}
    \frac{H^2}{\max\{p,q\}}>T^{1+\epsilon},
\end{equation}
which follows naturally from requiring the main term to dominate the error term. We comment in Remark \ref{remark1} how our lower bound on $H$ arises, and on the possibility of improvement.

Finally, as a generalization of Corollary 2 in \cite{khan25} and reminiscent of Conrey's reciprocity relation, we can also obtain the following Corollary from Theorem \ref{thm1}.
\begin{cor} \label{cor1}
    Let $p,q$ be distinct odd primes, $T>1$ and $H=T^{\delta}$ for $\delta \in \lc \h,1 \rc$. We have
    \begin{equation} \label{eqcor1}
    \begin{aligned}
        \sumstar_{\substack{\chi \bmod p}}& \frac{\sqrt{p}}{p-1} \chi(q) \inftyInt G_{T,H}\lc \tfrac{1}{2}+it \rc \frac{\Gamma\lc \frac{1+2\alpha-2it}{4} \rc}{\Gamma\lc \frac{1+2\alpha+2it}{4} \rc} \lc \frac{\pi}{q} \rc^{it}\lv L\lc \tfrac{1}{2}+it,\chi \rc \rv^2 dt = \\
        &\sumstar_{\substack{\chi \bmod q}} \frac{\sqrt{q}}{q-1} \chi(-p) \inftyInt G_{T,H}\lc \tfrac{1}{2}+it \rc \frac{\Gamma\lc \frac{1+2\alpha-2it}{4} \rc}{\Gamma\lc \frac{1+2\alpha+2it}{4} \rc} \lc \frac{\pi}{p} \rc^{it}\lv L\lc \tfrac{1}{2}+it,\chi \rc \rv^2 dt \\
        &+ O\lc \frac{T}{H}(pqT)^{\epsilon} \left[\lc \frac{p}{q} \rc^{\h} + \lc \frac{q}{p} \rc^{\h} \right] \rc,
    \end{aligned}
    \end{equation}
    for $\alpha= 0$ if $\chi(-1)=1$ and $\alpha=1$ if $\chi(-1)=-1$.
\end{cor}

\subsection{Notations}
Overall in the paper, we set $e(x)=\exp(2\pi ix)$, $T>1$, $p,q$ distinct odd primes, and $H=T^{\delta}$ for $\delta \in \lc \h,1 \rc$. $\sumstar$ denotes the sum over primitive Dirichlet characters, and $\tau(\chi)$ denotes the Gauss sum $\sum\limits_{n \bmod{q}} \chi(n)e(\tfrac{n}{q})$ for a Dirichlet character $\chi$ modulo $q$. Moreover, $\varepsilon$ denotes an arbitrarily small positive constant. 

\section{Preliminaries}
\begin{lem} \label{intbound}
    For $\epsilon >0$ be an arbitrary small, we have
    \begin{equation}
        \inftyInt  \lv \frac{\zeta\lc\tfrac{1}{2}-it\rc}{-\tfrac{1}{2}+it} \rv^2 \exp\lc-\frac{(t-T)^2}{H^2}\rc dt = O\lc \frac{H}{T^2}T^{\epsilon} \rc.
    \end{equation}
    \begin{proof}
        First, we split the integral at $|t-T|=H^{1+\epsilon}$. Then, for $|t-T|\leq H^{1+\epsilon}$, we have $|t| \asymp T$ and
        \begin{equation}
            \int_{|t-T|\leq H^{1+\epsilon}}\lv \frac{\zeta\lc\tfrac{1}{2}-it\rc}{-\tfrac{1}{2}+it} \rv^2 \exp\lc-\frac{(t-T)^2}{H^2}\rc dt \ll \frac{1}{T^2}\int_{|t-T|\leq H^{1+\epsilon}}\lv \zeta\lc\tfrac{1}{2}-it\rc \rv^2 \exp\lc-\frac{(t-T)^2}{H^2}\rc dt
        \end{equation}
        And, by the results in \cite{Atk49} and \cite{Jut83}, we have
        \begin{equation}
            \int_{|t-T|\leq H^{1+\epsilon}}\lv \zeta\lc\tfrac{1}{2}-it\rc \rv^2 \exp\lc-\frac{(t-T)^2}{H^2}\rc dt \ll \int_{|t-T|\leq H^{1+\epsilon}}\lv \zeta\lc\tfrac{1}{2}-it\rc \rv^2 dt \ll HT^{\epsilon}.
        \end{equation}
        Then
        \begin{equation}
            \int_{|t-T|\leq H^{1+\epsilon}}\lv \frac{\zeta\lc\tfrac{1}{2}-it\rc}{-\tfrac{1}{2}+it} \rv^2 \exp\lc-\frac{(t-T)^2}{H^2}\rc dt \ll \frac{H}{T^2}T^{\epsilon}.
        \end{equation}
        For the case $|t-T|>H^{1+\epsilon}$, we have using Lemma \ref{conv}
        \begin{equation}
        \begin{aligned}
            \int_{|t-T| > H^{1+\epsilon}}\lv \frac{\zeta\lc\tfrac{1}{2}-it\rc}{-\tfrac{1}{2}+it} \rv^2 \exp\lc-\frac{(t-T)^2}{H^2}\rc dt &\ll \int_{|t-T| > H^{1+\epsilon}} \frac{|t|^{\h+ \epsilon}}{1+t^2}  \exp\lc-H^{\epsilon}\rc dt\\
            &\ll \exp\lc-H^{\epsilon}\rc\int_{|t-T| > H^{1+\epsilon}} (1+|t|)^{-\frac{3}{2}+\epsilon} dt\\
            &\ll \exp\lc-H^{\epsilon}\rc.
        \end{aligned}
        \end{equation}
    \end{proof}
\end{lem}

\begin{lem} \label{bound}
    Let $f_{T,H}(y) = 1$ or $f_{T,H}(y) = Hy-iT$, $\psi_{s,T,H}(y)=-y^2 + \frac{2s-1}{H}y + i\frac{2T}{H}y$, and $\phi_{n}(y)=\frac{2\pi n q}{p}e^{\frac{2y}{H}}$. Then, the following expression
    \begin{equation}\label{lem1eq1}
        \mathcal{H}_{n,m}(s)=\infS \frac{1}{nm} \inftyInt f_{T,H}(y)\exp\lc \psi_{s,T,H}(y) \rc \sin \lc \phi_{nm}(y) \rc dy
    \end{equation}
    is defined and analytic on $D = \{s \in \mathbb{C} |\mathrm{Re}(s) \in \left( -\frac{1}{10}, \frac{3}{2} \right),\mathrm{Im}(s) \in \left( -1,1 \right) \}$. Moreover, the sum can be restricted to $n,m\leq T^{50}$ with an error at most $O\lc pq^{-1}T^{-47} \rc$. Furthermore, the expression
    \begin{equation}\label{lem1eq2}
        \mathcal{H}_n(s)=\sum_{n\geq 1} \frac{1}{n} \inftyInt f_{T,H}(y)\exp\lc \psi_{s,T,H}(y) \rc \sin \lc \phi_n(y) \rc dy
    \end{equation}
    is also analytic on the same domain $D$ and the sum can be restricted to $n\leq T^{50}$ with the error at most $O\lc pq^{-1}T^{-47} \rc$.
    \begin{proof}
        For $f_{T,H}(y) = Hy-iT$, set
        \begin{equation}
            I_{nm}(s)=\inftyInt f_{T,H}(y)\exp\lc \psi_{s,T,H}(y) \rc \sin \lc \phi_{nm}(y) \rc dy.
        \end{equation}
       We apply integration by parts on $I_{nm}(s)$ and get
        \begin{equation} \label{ibp}
        \begin{aligned}
            I_{nm}(s)&=-\inftyInt f_{T,H}(y)\exp\lc \psi_{s,T,H}(y) \rc \frac{1}{\phi_{nm}'(y)} \frac{d}{dy} \lc \cos \lc \phi_{nm}(y) \rc \rc dy\\
            &= \frac{pH}{4\pi nmq}\inftyInt \frac{d}{dy} \lc f_{T,H}(y)\exp\lc \psi_{s,T,H}(y) - \frac{2y}{H} \rc \rc \cos \lc \phi_{nm}(y) \rc dy\\
            &= \frac{pH}{4\pi nmq}\inftyInt P_{s,T,H}(y) \exp\lc \psi_{s,T,H}(y) - \frac{2y}{H} \rc \cos \lc \phi_{nm}(y) \rc dy,
        \end{aligned}
        \end{equation}
        where $P_{s,T,H}(y)=-2Hy^2+(2s-3+4iT)y+\frac{2T^2}{H}+H-i\frac{T}{H}(2s-3)$. Then, 
        \begin{equation} \label{eq1}
            \infS \lv \frac{1}{nm} I_{nm}(s) \rv \ll \infS \frac{pH}{q(nm)^2} \inftyInt \lv P_{s,T,H}(y) \rv \exp\lc -y^2 + \frac{2\text{Re}(s)-3}{H}y \rc dy.
        \end{equation}
        For $s \in D$, the coefficients of $P_{s,T,H}(y)$ is at most $T^2$. Thus, we have 
        \begin{equation}
            \inftyInt \lv P_{s,T,H}(y) \rv \exp\lc -y^2 + \frac{2\text{Re}(s)-3}{H}y \rc dy \ll T^2
        \end{equation}
        and
        \begin{equation}
            \infS \lv\frac{1}{nm} I_{nm}(s)\rv\ll \infS \frac{pHT^{2}}{q(nm)^2}  < \infty.
        \end{equation}
        This implies that $\mathcal{H}_{n,m}(s)=\sum\limits_{n,m\geq 1} \frac{1}{nm} I_{nm}(s)$ converges uniformly on $D$. Since $I_{nm}(s)$ is analytic on $D$ for each $n,m$, $\mathcal{H}_{n,m}(s)$ is analytic on $D$. We now restrict the sum to $n,m\leq T^{50}$. Decompose the sum into
        \begin{equation}
             \infS \lv\frac{1}{nm} I_{nm}(s)\rv = \sum_{n,m \leq T^{50}} \lv\frac{1}{nm} I_{nm}(s)\rv + 2\sum_{\substack{n \leq T^{50}\\ m>T^{50}}} \lv\frac{1}{nm} I_{nm}(s)\rv+ \sum_{n,m > T^{50}} \lv\frac{1}{nm} I_{nm}(s)\rv.
        \end{equation}
        Then, it remains to bound the last two terms. By \eqref{ibp}, we have
        \begin{equation}
            2\sum_{\substack{n \leq T^{50}\\ m>T^{50}}} \lv\frac{1}{nm} I_{nm}(s)\rv+ \sum_{n,m > T^{50}} \lv\frac{1}{nm} I_{nm}(s)\rv \ll HT^{2}\frac{p}{q} \lc \sum_{\substack{n \leq T^{50}\\ m>T^{50}}} \frac{1}{(nm)^2}+ \sum_{n,m > T^{50}} \frac{1}{(nm)^2} \rc.
        \end{equation}
        Since $\sum_{n\leq T^{50}} n^{-2} = O(1)$ and
        \begin{equation}
            \sum_{n> T^{50}} \frac{1}{n^2} \ll \int_{T^{50}}^{\infty}\frac{dt}{t^2} = \frac{1}{T^{50}},
        \end{equation}
        we have
        \begin{equation}
            2\sum_{\substack{n \leq T^{50}\\ m>T^{50}}} \lv\frac{1}{nm} I_{nm}(s)\rv + \sum_{n,m > T^{50}} \lv\frac{1}{nm} I_{nm}(s)\rv \ll HT^{2}\frac{p}{q}\lc \frac{1}{T^{50}} \rc \ll \frac{p}{qT^{47}}.
        \end{equation}
        A similar argument applies to $\mathcal{H}_n(s)$, restricting the sum to $n \le T^{50}$ with error $O(pq^{-1}T^{-47})$.
        
        Finally, when $f_{T,H}(y) = 1$, we obtain
        \begin{equation}
            P_{s,T,H}(y)=-2y+\frac{2s-3}{H}+i\frac{2T}{H}
        \end{equation}
        from the integration by parts step, and the same conclusion follows from the argument above.
    \end{proof}
\end{lem}

\begin{lem} \label{bound2}
    Let $P_j(y)$ be a polynomial of degree at most $j\geq1$, $\Delta_{\pm} = \frac{T}{2\pi}\pm \frac{nmq}{p}$, and
    \begin{equation}
        \mathcal{H}_{\pm}=\sum_{n,m\leq T^{50}} \frac{1}{nm} \inftyInt P_{j}(y) \exp \lc -y^2-\frac{y}{H} \rc e\lc \frac{Ty}{H\pi}\pm \frac{nmq}{p}e^{\frac{2y}{H}}\rc dy,
    \end{equation}
    then the sum in $\mathcal{H}_{\pm}$ can be restricted to $|\Delta_{\pm}| \leq H(pqH)^{\epsilon}$ with an error of $O((pqT)^{\epsilon}H^{-1})$. Moreover,
    \begin{equation}
        |\mathcal{H}_{+}| \ll (pqT)^{\epsilon}H^{-1}.
    \end{equation}
    \begin{proof}
        Define
        \begin{equation}
            I_{nm} = \inftyInt P_{j}(y) \exp \lc -y^2-\frac{y}{H} \rc e\lc \frac{Ty}{H\pi}\pm \frac{nmq}{p}e^{\frac{2y}{H}}\rc dy
        \end{equation}
        Restricting the integral to $|y|<(pqH)^{\epsilon}$ with a negligible error of $O\lc e^{-(pqH)^{\epsilon}} \rc$, we have
        \begin{equation}
            \exp\lc -\frac{y}{H}\rc = 1 + O\lc \frac{(pqH)^{\epsilon}}{H} \rc
        \end{equation}
        \begin{equation}
            \exp\lc \frac{2y}{H} \rc = 1 + \frac{2y}{H} + R_H(y)
        \end{equation}
        for
        \begin{equation}
            R_H(y) = \sum_{l=2}^{\infty} \lc\frac{2y}{H}\rc^l \frac{1}{l!}.
        \end{equation}
        Then
        \begin{equation} \label{l2eq1}
            I_{nm} = \int_{\lv y \rv \leq (pqH)^{\epsilon}} P_{j}(y) \exp \lc -y^2\rc e\lc \frac{Ty}{H\pi}\pm \frac{nmq}{p}\lc 1 + \frac{2y}{H} + R_H(y) \rc\rc dy + O\lc \frac{(pqH)^{\epsilon}}{H} \rc.
        \end{equation}
        Let
        \begin{equation} \label{stepA}
            \phi(y) = i\lc\frac{2Ty}{H}\pm \frac{2\pi nmq}{p}\lc 1 + \frac{2y}{H}\rc\rc
        \end{equation}
        and differentiate to get
        \begin{equation}
            \phi'(y)=\frac{4\pi i}{H}\lc\frac{T}{2\pi}\pm \frac{nmq}{p}\rc = \frac{4\pi i}{H}\Delta_{\pm}.
        \end{equation}
        Then we get
        \begin{equation}
            I_{nm} = \int_{\lv y \rv \leq (pqH)^{\epsilon}} P_{j}(y) \exp \lc -y^2 \pm \frac{2\pi i nmq}{p}R_H(y)\rc \exp\lc \phi(y)\rc dy + O\lc \frac{(pqH)^{\epsilon}}{H} \rc.
        \end{equation}
        Suppose $|\Delta_{\pm}| > H(pqH)^{\epsilon}$, apply the integration by part, and get
        \begin{equation}\label{restr}
        \begin{aligned}
            I_{nm} &= \int_{\lv y \rv \leq (pqH)^{\epsilon}} P_{j}(y) \exp \lc -y^2 \pm \frac{2\pi i nmq}{p}R_H(y)\rc \frac{1}{\phi'(y)} \frac{d}{dy} \exp\lc \phi(y)\rc dy + O\lc \frac{(pqH)^{\epsilon}}{H} \rc\\
            &= \frac{H}{4\pi i \Delta_{\pm}}\int_{\lv y \rv \leq (pqH)^{\epsilon}} P_{j}(y) \exp \lc -y^2 \pm \frac{2\pi i nmq}{p}R_H(y)\rc \frac{d}{dy} \exp\lc \phi(y)\rc dy + O\lc \frac{(pqH)^{\epsilon}}{H}\rc\\
            &= \frac{H}{4\pi i \Delta_{\pm}} \left[ \left. P_{j}(y) \exp \lc -y^2 \pm \frac{2\pi i nmq}{p}R_H(y)\rc \exp\lc \phi(y)\rc \right|_{-(pqH)^{\epsilon}}^{(pqH)^{\epsilon}} \right. \\
            &- \left. \int_{\lv y \rv \leq (pqH)^{\epsilon}} \exp\lc \phi(y)\rc \frac{d}{dy} \lc P_{j}(y) \exp \lc -y^2 \pm \frac{2\pi i nmq}{p}R_H(y)\rc \rc dy \right] + O\lc \frac{(pqH)^{\epsilon}}{H} \rc.
        \end{aligned}
        \end{equation}
        Observe that
        \begin{equation}
            \frac{H}{4\pi i \Delta_{\pm}} \left[ \left. P_{j}(y) \exp \lc -y^2 \pm \frac{2\pi i nmq}{p}R_H(y)\rc \exp\lc \phi(y)\rc \right|_{-(pqH)^{\epsilon}}^{(pqH)^{\epsilon}} \right] = O\lc e^{-(pqH)^{\epsilon}} \rc
        \end{equation}
        can be absorbed into the error term, and the integral is
        \begin{equation}
        \begin{aligned}
            &\int_{\lv y \rv \leq (pqH)^{\epsilon}} \exp\lc \phi(y)\rc \frac{d}{dy} \lc P_{j}(y) \exp \lc -y^2 \pm \frac{2\pi i nmq}{p}R_H(y)\rc \rc dy \\
            &= \int_{\lv y \rv \leq (pqH)^{\epsilon}} \exp\lc \phi(y)\rc \exp \lc -y^2 \pm \frac{2\pi i nmq}{p}R_H(y)\rc\lc P_{j-1}(y) + P_{j}(y) \lc -2y \pm \frac{2\pi i nmq}{p}R_H'(y) \rc \rc dy \\
            &= \int_{\lv y \rv \leq (pqH)^{\epsilon}} \exp \lc -y^2 \rc e\lc \frac{Ty}{H\pi}\pm \frac{nmq}{p}e^{\frac{2y}{H}}\rc \lc P_{j+1}(y) \pm \frac{2\pi i nmq}{p}P_j(y)R_H'(y) \rc dy.
        \end{aligned}
        \end{equation}
        Then we have
        \begin{equation} \label{stepZ}
        \begin{aligned}
            I_{nm} =&- \frac{H}{4\pi i \Delta_{\pm}} \left[ \int_{\lv y \rv \leq (pqH)^{\epsilon}} \exp \lc -y^2 \rc e\lc \frac{Ty}{H\pi}\pm \frac{nmq}{p}e^{\frac{2y}{H}}\rc \lc P_{j+1}(y) \pm \frac{2\pi i nmq}{p}P_j(y)R_H'(y) \rc dy \right]\\
            &+ O\lc \frac{(pqH)^{\epsilon}}{H} \rc.
        \end{aligned}
        \end{equation}
        Observe that this contains the expression
        \begin{equation} \label{expl2}
            \int_{\lv y \rv \leq (pqH)^{\epsilon}} \exp \lc -y^2 \rc e\lc \frac{Ty}{H\pi}\pm \frac{nmq}{p}e^{\frac{2y}{H}}\rc P_{j+1}(y)dy,
        \end{equation}
        which maintains the form of \eqref{l2eq1}. So we can repeat the arguments from \eqref{stepA} to \eqref{stepZ} for $k = \lceil\frac{100}{\epsilon}\rceil$ times and obtain
        \begin{equation} \label{stepB}
        \begin{aligned}
            I_{nm} &= \lc - \frac{H}{4\pi i \Delta_{\pm}}\rc^{k} \int_{\lv y \rv \leq (pqH)^{\epsilon}} \exp \lc -y^2 \rc e\lc \frac{Ty}{H\pi}\pm \frac{nmq}{p}e^{\frac{2y}{H}}\rc P_{j+k}(y)dy\\
            &+ \sum_{l=1}^{k} \lc \mp \frac{H}{4\pi i \Delta_{\pm}} \rc^{l} I_{nm,l} + O\lc \frac{(pqH)^{\epsilon}}{H} \rc
        \end{aligned}
        \end{equation}
        for
        \begin{equation}
             I_{nm,l} = \int_{\lv y \rv \leq (pqH)^{\epsilon}} \exp \lc -y^2 \rc e\lc \frac{Ty}{H\pi} \pm \frac{nmq}{p}e^{\frac{2y}{H}}\rc \frac{2\pi i nmq}{p}P_{j+l-1}(y)R_H'(y) dy.
        \end{equation}
        First, we bound the first term in \eqref{stepB}. Applying $|\Delta_{\pm}| > H(pqH)^{\epsilon}$, we must have
        \begin{equation}
        \begin{aligned}
            \left|\lc - \frac{H}{4\pi i \Delta_{\pm}}\rc^{k} \int_{\lv y \rv \leq (pqH)^{\epsilon}} \exp \lc -y^2 \rc \right. & \left. e\lc \frac{Ty}{H\pi}\pm \frac{nmq}{p}e^{\frac{2y}{H}}\rc P_{j+k}(y)dy\right|\\
            &\ll_k\frac{1}{(pqH)^{k\epsilon}} \int_{\lv y \rv \leq (pqH)^{\epsilon}} |P_{j+k}(y)| \exp \lc -y^2 \rc dy \\
            &\ll_k \frac{1}{(pqH)^{100}}.     
        \end{aligned}
        \end{equation}
        Then the sum of the integral contributes
        \begin{equation}\label{l2eq2}
            \lv \sum_{\substack{n,m\leq T^{50} \\ |\Delta_{\pm}| > H(pqH)^{\epsilon}}} \frac{1}{nm} \inftyInt P_{j}(y) \exp \lc -y^2\rc e\lc \frac{Ty}{H\pi}\pm \frac{nmq}{p}e^{\frac{2y}{H}}\rc dy \rv \ll_k  \frac{1}{(pqH)^{100}}\sum_{n,m\leq T^{50}} \frac{1}{nm}\ll_k \frac{ T^{\epsilon}}{(pqH)^{100}}.
        \end{equation}
        Now, we need to show the remaining terms are $O\lc \frac{(pqT)^{\epsilon}}{H} \rc$. Observe that $R_H'(y) = O\lc \frac{y}{H^2} \rc$. Then, for each $l>0$, we have
        \begin{equation}
            |I_{nm,l}| \ll \int_{\lv y \rv \leq (pqH)^{\epsilon}} \exp \lc -y^2 \rc \frac{nmq}{p}\lv P_{j+l-1}(y) \rv \frac{\lv y \rv}{H^2} dy \ll \frac{nmq}{pH^2}.
        \end{equation}
        Thus, substituting into the sum, we obtain
        \begin{equation}
        \begin{aligned}
            \lv\sum_{\substack{n,m\leq T^{50} \\ |\Delta_{\pm}| > H(pqH)^{\epsilon}}} \frac{1}{nm} \sum_{l=1}^{k} \lc \mp \frac{H}{4\pi i \Delta_{\pm}} \rc^{l} I_{nm,l}\rv &\ll \sum_{\substack{n,m\leq T^{50} \\ |\Delta_{\pm}| > H(pqH)^{\epsilon}}} \frac{1}{nm} \sum_{l=1}^{k} \lc\frac{H}{\lv \Delta_{\pm} \rv}\rc^l \frac{nmq}{pH^2} \\
            &= \frac{q}{pH^2}\sum_{l=1}^{k} H^l \sum_{\substack{n,m\leq T^{50} \\ |\Delta_{\pm}| > H(pqH)^{\epsilon}}} \frac{1}{\lv \Delta_{\pm} \rv^l}.
        \end{aligned}
        \end{equation}
        Let $r=nm$. Then the number of terms $nm=r$ is at most $d(r)\ll r^{\epsilon} \ll T^{\epsilon}$, and we obtain
        \begin{equation}
            \sum_{\substack{n,m\leq T^{50} \\ |\Delta_{\pm}| > H(pqH)^{\epsilon}}} \frac{1}{\lv \Delta_{\pm} \rv^l} \ll \sum_{\substack{r\leq T^{100} \\ |\frac{T}{2\pi}\pm \frac{rq}{p}| > H(pqH)^{\epsilon}}} \frac{T^{\epsilon}}{\lv\frac{T}{2\pi}\pm \frac{rq}{p}\rv^l} =\lc\frac{p}{q}\rc^l\sum_{\substack{r\leq T^{100} \\ |\frac{Tp}{2\pi q}\pm r| > \frac{pH}{q}(pqH)^{\epsilon}}} \frac{T^{\epsilon}}{\lv\frac{Tp}{2\pi q}\pm r\rv^l}.
        \end{equation}
        For $N_r \leq |\frac{pT}{2\pi q} \pm r| \leq N_r+1$, we have $\frac{1}{\lv\frac{pT}{2\pi q}- r\rv} \leq \frac{1}{N_r}$ and
        \begin{equation}
             \lc\frac{p}{q}\rc^l\sum_{\substack{r\leq T^{100} \\ |\frac{Tp}{2\pi q}\pm r| > \frac{pH}{q}(pqH)^{\epsilon}}} \frac{T^{\epsilon}}{\lv\frac{Tp}{2\pi q}\pm r\rv^l} \ll \lc\frac{p}{q}\rc^l\sum_{\frac{pH}{q}(pqH)^{\epsilon} < N_r < T^{101} }\frac{T^{\epsilon}}{N_r^l}.
        \end{equation}
        Now, for $l=1$, we have
        \begin{equation}
            \lc\frac{p}{q}\rc\sum_{\frac{pH}{q}(pqH)^{\epsilon} < N_r < T^{101} } \frac{T^{\epsilon}}{N_r} \ll \lc\frac{pT^{\epsilon}}{q}\rc \sum_{1 \leq N_r < T^{101} } \frac{1}{N_r} \ll \lc\frac{pT^{\epsilon}}{q}\rc T^{\epsilon} \ll \frac{pT^{\epsilon}}{q}.
        \end{equation}
        For $l\geq2$, we have
        \begin{equation}
            \lc\frac{p}{q}\rc^l\sum_{\frac{pH}{q}(pqH)^{\epsilon} < N_r < T^{101} } \frac{T^{\epsilon}}{N_r^l} \ll \lc\frac{p}{q}\rc^l T^{\epsilon} \sum_{N_r > \frac{pH}{q}(pqH)^{\epsilon}} \frac{1}{N_r^l} \ll \lc\frac{p}{q}\rc^l T^{\epsilon} \int_{\frac{pH(pqH)^{\epsilon}}{q}}^{\infty} \frac{1}{x^l}dx \ll \frac{p}{q} \frac{(pqT)^{\epsilon}}{H^{l-1}}.
        \end{equation}
        Thus, we have
        \begin{equation}
        \begin{aligned}
            \frac{q}{pH^2}\sum_{l=1}^{k} H^l \sum_{\substack{n,m\leq T^{50} \\ |\Delta_{\pm}| > H(pqH)^{\epsilon}}} \frac{1}{\lv \Delta_{\pm} \rv^l} &=\frac{q}{pH^2} \lc H \sum_{\substack{n,m\leq T^{50} \\ |\Delta_{\pm}| > H(pqH)^{\epsilon}}} \frac{1}{\lv \Delta_{\pm} \rv} + \sum_{l=2}^{k} H^l \sum_{\substack{n,m\leq T^{50} \\ |\Delta_{\pm}| > H(pqH)^{\epsilon}}} \frac{1}{\lv \Delta_{\pm} \rv^l}\rc\\
            &\ll \frac{q}{pH^2} \lc H\frac{pT^{\epsilon}}{q} + \sum_{l=2}^{k} H^l \frac{p}{q} \frac{(pqT)^{\epsilon}}{H^{l-1}}\rc\\
            &= \frac{1}{H^2} \lc HT^{\epsilon} + \sum_{l=2}^{k} H(pqT)^{\epsilon}\rc \\
            &\ll_k \frac{(pqT)^{\epsilon}}{H}.
        \end{aligned}
        \end{equation}
        Now, for the error term $O\lc \frac{(pqH)^{\epsilon}}{H} \rc$, we have
        \begin{equation}
            \sum_{n,m\leq T^{50}} \frac{1}{nm} \frac{(pqH)^{\epsilon}}{H} \ll \frac{(pqT)^{\epsilon}}{H}.
        \end{equation}
        Finally, for $\mathcal{H}_{+}$, the restriction $|\Delta_{+}| \leq H(pqH)^{\epsilon}$ implies $\frac{nmq}{p} \asymp -T$, which is impossible for $n,m,p,q,T>0$. Thus, the bound applies to the entire sum for the $\mathcal{H}_+$ case.
    \end{proof}
\end{lem}

\begin{lem} \label{a1}
    Let $P_j(y)$ be a polynomial of degree at most $j\geq1$, $\Delta_{nm} = \frac{T}{2\pi}- \frac{nmq}{p}$, and
    \begin{equation} \label{l3eq1}
        \mathcal{H} = \sum_{\lv \Delta_{nm} \rv \leq H(pqH)^{\epsilon}}\frac{1}{nm} \inftyInt P_j(y) \exp \lc -y^2 - \frac{y}{H}\rc e \lc \frac{Ty}{H\pi} - \frac{nmq}{p}e^{\frac{2y}{H}} \rc dy.
    \end{equation}
    Then, we have
    \begin{equation}\label{l3eq2}
    \begin{aligned}
        \mathcal{H} =&\sum_{\lv \Delta_{nm} \rv \leq H(pqH)^{\epsilon}}\frac{1}{nm} \inftyInt P_j(y) \exp \lc -y^2 \rc e \lc \frac{Ty}{H\pi} - \frac{2nmqy}{pH} \rc e\lc -\frac{nmq}{p} \rc  dy + O\lc \frac{(pqH)^{\epsilon}}{H} \rc.
    \end{aligned}
    \end{equation}
    \begin{proof}
        Observe that $nm \asymp \frac{Tp}{q}$ and the sum of $nm$ lies in an interval of length $\frac{pH(pqH)^{\epsilon}}{q}$. Then we have
        \begin{equation} \label{l3eq3}
            \sum_{\lv \Delta_{nm} \rv \leq H(pqH)^{\epsilon}}\frac{1}{nm} \ll \lc \frac{pH(pqH)^{\epsilon}}{q} \rc \lc \frac{1}{\frac{Tp}{q}} \rc = \frac{H(pqH)^{\epsilon}}{T}.
        \end{equation}
        Restricting the integral to $|y| < (pqH)^{\epsilon}$ with a negligible error of $O\lc e^{-(pqH)^{\epsilon}} \rc$ and expanding $\exp\lc-\frac{y}{H}\rc$, we get
        \begin{equation}\label{l3eq4}
            \exp\lc -\frac{y}{H} \rc = 1 + O\lc \frac{y}{H} \rc = 1 + O\lc \frac{(pqH)^{\epsilon}}{H} \rc.
        \end{equation}
        Then we obtain
        \begin{equation} \label{l3eq5}
        \begin{aligned}
            \inftyInt P_j(y) & \exp \lc -y^2 - \frac{y}{H}\rc e \lc \frac{Ty}{H\pi} - \frac{nmq}{p}e^{\frac{2y}{H}} \rc dy \\
            &= \int_{\lv y \rv \leq (pqH)^{\epsilon}} P_j(y) \exp \lc -y^2\rc e \lc \frac{Ty}{H\pi} - \frac{nmq}{p}e^{\frac{2y}{H}} \rc dy + O\lc \frac{(pqH)^{\epsilon}}{H} \rc.
        \end{aligned}
        \end{equation}
        Substituting the error in \eqref{l3eq5} into \eqref{l3eq1}, we obtain an error term of size
        \begin{equation}
            \sum_{\lv \Delta_{nm} \rv \leq H(pqH)^{\epsilon}}\frac{1}{nm}\frac{(pqH)^{\epsilon}}{H} \ll \frac{H(pqH)^{\epsilon}}{T}\frac{(pqH)^{\epsilon}}{H} = \frac{(pqH)^\epsilon}{T}.
        \end{equation}
        Thus, we get
        \begin{equation} \label{l3eq6}
            \mathcal{H} = \sum_{\lv \Delta_{nm} \rv \leq H(pqH)^{\epsilon}}\frac{1}{nm} \int_{\lv y \rv \leq (pqH)^{\epsilon}} P_j(y) \exp \lc -y^2\rc e \lc \frac{Ty}{H\pi} - \frac{nmq}{p}e^{\frac{2y}{H}} \rc dy + O\lc \frac{(pqH)^\epsilon}{T} \rc.
        \end{equation}
        Since $\frac{nmq}{p} \asymp T$ and $\frac{T}{H^2}<1$, we obtain
        \begin{align} \label{taylor3}
            e \lc - \frac{nmq}{p}e^{\frac{2y}{H}} \rc =  e\lc -\frac{nmq}{p} \rc e\lc -\frac{2nmqy}{pH} \rc + O\lc \frac{T(pqH)^{\epsilon}}{H^2} \rc.
        \end{align}
        Then we get
        \begin{equation} \label{l3eq7}
        \begin{aligned}
            \int_{\lv y \rv \leq (pqH)^{\epsilon}} P_j(y) & \exp \lc -y^2\rc e \lc \frac{Ty}{H\pi} - \frac{nmq}{p}e^{\frac{2y}{H}} \rc dy\\
            &= \int_{\lv y \rv \leq (pqH)^{\epsilon}} P_j(y) \exp \lc -y^2\rc e \lc \frac{Ty}{H\pi} -\frac{2nmqy}{pH} \rc e\lc -\frac{nmq}{p} \rc dy + O\lc \frac{T(pqH)^{\epsilon}}{H^2} \rc
        \end{aligned}
        \end{equation}
        Substituting \eqref{l3eq7} into \eqref{l3eq6} and expanding the integral back to infinity with a negligible error, we obtain
        \begin{equation}
            \mathcal{H} = \sum_{\lv \Delta_{nm} \rv \leq H(pqH)^{\epsilon}}\frac{1}{nm} \inftyInt P_j(y) \exp \lc -y^2\rc e \lc \frac{Ty}{H\pi} -\frac{2nmqy}{pH} \rc e\lc -\frac{nmq}{p} \rc dy + O\lc \frac{(pqH)^\epsilon}{H} \rc.
        \end{equation}
    \end{proof}
\end{lem}

\begin{lem}  [Mellin Inversion] \cite[Chapter III, Theorem 2]{Won01} \label{mellininversion}
    Suppose that $f:(0,\infty) \xrightarrow{}\mathbb{R}$ is continuous and the integral
    \begin{equation}
        \phi(w) = \int_{0}^{\infty} x^{w-1} f(x)dx
    \end{equation}
    is absolutely convergent for $\mathrm{Re}(s) \in (a,b)$. Then
    \begin{equation}
        f(x) = \frac{1}{2\pi i} \int_{(c)} x^{-w} \phi(w)dw
    \end{equation}
    for $c \in (a,b)$
\end{lem}

\begin{lem}\label{mellin2}
    For $\mathrm{Re}(w)=\sigma>0$, $\GT(w)$,defined in \eqref{GTH}, is absolutely convergent. Moreover
    \begin{equation}\label{GSigma}
        \GT(\sigma)\asymp HT^{\sigma-1}
    \end{equation}
    for fixed $\sigma>0$. In particular,
    \begin{equation}\label{G1}
        \GT(1) \asymp H,
    \end{equation}
    and
    \begin{equation}\label{G'1}
        \GT'(1)\asymp H\log T.
    \end{equation}
\end{lem}

\begin{proof}
    To show $\GT(w)$ is absolutely convergent for $\sigma>0$, it is enough to show
    \begin{equation}
        |\GT(w)| \leq \int_{0}^{\infty} e^ {-\frac{(2\pi x - T)^2}{H^2}} x^{\sigma-1} dx < \infty.
    \end{equation}
    For $x \in (1,\infty)$, we have
    \begin{equation}
        \int_{1}^{\infty} e^ {-\frac{(2\pi x - T)^2}{H^2}} x^{\sigma-1} dx \ll \int_{1}^{\infty} e^ {-\frac{(2\pi x - T)^2}{H^2}} x^{\lv\sigma-1\rv} dx < \infty.
    \end{equation}
    For $x \in (0,1]$, we have
    \begin{equation}
        \int_{0}^{1} e^ {-\frac{(2\pi x - T)^2}{H^2}} x^{\sigma-1} dx \ll \int_{0}^{1} x^{\sigma-1} dx = \frac{1}{\sigma} < \infty
    \end{equation}
    for $\sigma>0$. Now, to prove \eqref{GSigma}, let $y=(2\pi x-T)/H$. Then
    \begin{equation}\label{Gy}
        \GT(w)=\frac{H}{2\pi}\int_{-T/H}^{\infty}e^{-y^2}\left(\frac{T+Hy}{2\pi}\right)^{w-1}dy.
    \end{equation}
    Let $w=\sigma$, then we have
    \begin{equation}
        \GT(\sigma)=\frac{H}{2\pi}\lc\frac{T}{2\pi}\rc^{\sigma-1}\int_{-T/H}^{\infty}e^{-y^2}\left(1+\frac{Hy}{T}\right)^{\sigma-1}dy.
    \end{equation}
    We need to show
    \begin{equation}
        \int_{-T/H}^{\infty}e^{-y^2}\left(1+\frac{Hy}{T}\right)^{\sigma-1}dy \asymp 1.
    \end{equation}
    For the lower bound, we first observe that $1\leq \left(1+\frac{Hy}{T}\right) \leq 2$ for $y \in [0,1]$. Thus, we have
    \begin{equation}
        \left(1+\frac{Hy}{T}\right)^{\sigma-1} \geq
        \begin{cases}
        1 & \text{if }\sigma-1 \geq 0 \\
        2^{\sigma-1} & \text{if } \sigma-1 < 0,
        \end{cases}
    \end{equation}
    where both cases are constants for fixed $\sigma$. Then,
    \begin{equation}
        \int_{-T/H}^{\infty}e^{-y^2}\left(1+\frac{Hy}{T}\right)^{\sigma-1}dy \gg \int_{0}^{1}e^{-y^2}\left(1+\frac{Hy}{T}\right)^{\sigma-1}dy \gg \int_{0}^{1}e^{-y^2}dy \gg 1.
    \end{equation}
    For the upper bound, we split the interval of integral into $(0,\infty)$, $\left[ -\frac{T}{2H},0 \right]$, and $\left[ -\frac{T}{H}, -\frac{T}{2H} \right]$. In the first interval, we have $\left(1+\frac{Hy}{T}\right)^{\sigma-1} \leq \left(1+\frac{Hy}{T}\right)^{\lv\sigma-1\rv}$. So we obtain
    \begin{equation}
         \int_{0}^{\infty}e^{-y^2}\left(1+\frac{Hy}{T}\right)^{\sigma-1}dy \ll \int_{0}^{\infty}e^{-y^2}\left(1+\frac{Hy}{T}\right)^{\lv\sigma-1\rv} dy \ll 1.
    \end{equation}
    For $\left[ -\frac{T}{2H},0 \right]$, we have $1+\frac{Hy}{T} \in \left[ \h,1 \right]$. This implies the following
    \begin{equation}
        \left(1+\frac{Hy}{T}\right)^{\sigma-1} \leq
        \begin{cases}
        1 & \text{if }\sigma-1 \geq 0 \\
        \lc \h \rc^{\sigma-1} & \text{if } \sigma-1 < 0,
        \end{cases}
    \end{equation}
    where both cases are constants for fixed $\sigma$. Thus
    \begin{equation}
        \int_{-T/2H}^{0}e^{-y^2}\left(1+\frac{Hy}{T}\right)^{\sigma-1}dy \ll \int_{-T/2H}^{0}e^{-y^2} dy \ll 1.
    \end{equation}
    Now, for $\left[ -\frac{T}{H}, -\frac{T}{2H} \right]$,let $u=1+\frac{H}{T}y$. Then $y = \frac{T}{H}(u-1)$ and $dy = \frac{T}{H}du$. So we obtain
    \begin{equation} \label{cov1}
    \begin{aligned}
        \int_{-T/H}^{-T/2H}e^{-y^2}\lc 1+\frac{Hy}{T} \rc^{\sigma-1} dy  &= \frac{T}{H} \int_{0}^{1/2} e^{-\frac{T^2}{H^2}(u-1)^2} u^{\sigma-1} du.
    \end{aligned}
    \end{equation}
    Observe that $-(u-1)^2 \leq -\frac{1}{4}$. Then we obtain
    \begin{equation} \label{cov2}
        \frac{T}{H} \int_{0}^{1/2} e^{-\frac{T^2}{H^2}(u-1)^2}  u^{\sigma-1} du \ll \frac{T}{H}e^{-\frac{T^2}{4H^2}} \int_{0}^{1/2}  u^{\sigma-1} du = \frac{T}{\sigma H}e^{-\frac{T^2}{4H^2}} \lc\h\rc^{\sigma} \ll 1.
    \end{equation}
    For \eqref{G1}, it follows from substituting $\sigma=1$ into \eqref{GSigma}. Finally, for \eqref{G'1}, we have
    \begin{equation}
    \begin{aligned}
        \GT'(1) &= \frac{H}{2\pi}\int_{-T/H}^{\infty}e^{-y^2} \log \lc \frac{T+Hy}{2\pi} \rc dy\\ &= \frac{H}{2\pi}\lc\int_{-T/H}^{\infty}e^{-y^2} \log \lc \frac{T}{2\pi} \rc dy + \int_{-T/H}^{\infty}e^{-y^2} \log \lc 1+\frac{Hy}{T} \rc dy \rc.
    \end{aligned}
    \end{equation}
    Observe that
    \begin{equation}
        \int_{-T/H}^{\infty}e^{-y^2} \log \lc \frac{T}{2\pi} \rc dy \asymp \log T.
    \end{equation}
    Thus, we just need to show the second integral is $O(1)$. For $y \in \lc 0,\infty \rc$, we have
    \begin{equation}
        \int_{0}^{\infty}e^{-y^2} \lv \log \lc 1+\frac{Hy}{T} \rc \rv dy \ll \int_{0}^{\infty}e^{-y^2} \frac{Hy}{T} dy \ll \frac{H}{T}\ll 1.
    \end{equation}
    For $y \in \left[ -\frac{T}{2H},0 \right]$, we have $\lv \log \lc 1+\frac{Hy}{T} \rc \rv \leq \log 2$. Thus
    \begin{equation}
        \int_{-T/2H}^{0}e^{-y^2}\lv \log \lc 1+\frac{Hy}{T} \rc \rv dy \ll  \int_{-T/2H}^{0}e^{-y^2} dy \ll 1
    \end{equation}
    For $y \in \left[ -\frac{T}{H}, -\frac{T}{2H} \right]$, we apply the same change of variable in \eqref{cov1}. Then we obtain
    \begin{equation}
    \begin{aligned}
        \int_{-T/H}^{-T/2H}e^{-y^2}\lv \log \lc 1+\frac{Hy}{T} \rc \rv dy &= \frac{T}{H} \int_{0}^{1/2} e^{-\frac{T^2}{H^2}(u-1)^2} \lv \log u \rv du.
    \end{aligned}
    \end{equation}
    With $-(u-1)^2 \leq -\frac{1}{4}$, we get
    \begin{equation}
        \frac{T}{H} \int_{0}^{1/2} e^{-\frac{T^2}{H^2}(u-1)^2} \lv \log u \rv du \ll \frac{T}{H}e^{-\frac{T^2}{4H^2}} \int_{0}^{1/2} \lv \log u \rv du = \frac{T}{H}e^{-\frac{T^2}{4H^2}} \lc \frac{\log 2}{2} + \h\rc \ll 1.
    \end{equation}
\end{proof}

\begin{lem} \label{weight}
    Let $\mathrm{Re}(w)=\sigma>0$ fixed. Then for every integer $n>0$, we have 
    \begin{equation}
        \lv G_{T,H}(w) \rv \ll_{n} HT^{\sigma -1}\lc \frac{\lc\frac{T}{H}\rc^n}{(1+|t|)^n} \rc.
    \end{equation}
    In particular, if $\lv t\rv \gg \frac{T}{H}$, then $G_{T,H}(w)$ decays faster than any fixed power of $\frac{\lc\frac{T}{H}\rc}{|t|}$.
    \begin{proof}
        let $y = \frac{2\pi x - T}{H}$, then we obtain $\frac{d}{dx}=\frac{2\pi}{H}\frac{d}{dy}$ and
        \begin{equation}
            \frac{d^n}{dx^n} \exp\lc -\frac{(2\pi x - T)^2}{H^2} \rc = \lc\frac{2\pi}{H}\rc^n\frac{d^n}{dy^n}\exp(-y^2) = \lc\frac{2\pi}{H}\rc^nP_n(y)\exp(-y^2),
        \end{equation}
        for $P_n(y)$ is a polynomial of degree $n$. Back to $G\lc \sigma + it \rc$, we apply integration by part $n$ times and get
        \begin{equation*}
        \begin{aligned}
            G_{T,H}\lc \sigma+it \rc &= -\int_{0}^{\infty} \frac{x^{\sigma + it -1 +1}}{\sigma -1 + it +1} \frac{d}{dx} \exp\lc -\frac{(2\pi x - T)^2}{H^2} \rc dx\\
            &=(-1)^{n} \int_{0}^{\infty} \frac{x^{\sigma + it -1 +n}}{\prod_{k=1}^{n}\lc \sigma + it -1 +k\rc} \frac{d^n}{dx^n} \exp\lc -\frac{(2\pi x - T)^2}{H^2} \rc dx\\
            &=(-1)^{n} \int_{-\frac{T}{H}}^{\infty} \frac{\lc \frac{T+Hy}{2\pi} \rc^{\sigma + it -1  +n}}{\prod_{k=1}^{n}\lc \sigma + it -1 +k\rc} \lc\frac{2\pi}{H}\rc^n P_n\lc y \rc\exp\lc -y^2 \rc \frac{H}{2\pi} dy.
        \end{aligned}
        \end{equation*}
        Since $\sigma-1+n >0$, we have 
        \begin{equation}
            \lv(T+Hy)^{\sigma-1+n}\rv < T^{\sigma - 1 + n}(1+|y|)^{\sigma - 1 + n}.
        \end{equation}
        Moreover, 
        \begin{equation}
            \sqrt{(\sigma-1+k)^2+t^2}\geq \frac{1}{\sqrt{2}}\lc \sigma-1+k +|t| \rc \gg 1+|t|
        \end{equation}
        for each $k$. Therefore
        \begin{align*}
            \lv G_{T,H}\lc \sigma + it \rc \rv &\ll_{n} \frac{H}{H^n}\int_{-\frac{T}{H}}^{\infty} \lv\frac{(T+Hy)^{\sigma-1+n}}{\prod_{k=1}^{n}\lc \sigma + it -1 +k\rc} P_n\lc y \rc \exp\lc -y^2 \rc \rv dy\\
            & \ll \frac{H}{H^n} \int_{-\frac{T}{H}}^{\infty} \frac{T^{\sigma - 1 + n}(1+|y|)^{\sigma - 1 + n}}{(1+|t|)^n} \lv P_n\lc y \rc \rv \exp\lc -y^2 \rc dy \\
            & = \frac{HT^{\sigma - 1 + n}}{H^n(1+|t|)^n} \int_{-\frac{T}{H}}^{\infty} (1+|y|)^{\sigma - 1 + n}\lv P_n\lc y \rc \rv \exp\lc -y^2 \rc dy \\
            & \ll HT^{\sigma - 1 }\lc \frac{\lc\frac{T}{H}\rc^n}{(1+|t|)^n} \rc.
        \end{align*}
    \end{proof}
\end{lem}

\begin{lem}[Functional equation] \cite[Chapter 9]{Dav80}\label{functional}
Let $\chi$ be a primitive character modulo $p$, then we have
\begin{equation}\label{fe}
  L(w,\chi)
  =\frac{\tau(\chi)}{i^{\alpha}\sqrt p}
   \left(\frac{p}{\pi}\right)^{\h-w}
   \frac{\Gamma\left(\frac{1-w+\alpha}{2}\right)}
        {\Gamma\left(\frac{w+\alpha}{2}\right)}
   L(1-w,\overline\chi),
\end{equation}
for $\alpha= 0$ if $\chi(-1)=1$ and $\alpha=1$ if $\chi(-1)=-1$.
\end{lem}

\begin{lem}[The convexity bound for the Riemann zeta function] \cite[Chapter V]{Tit86} \label{conv}
    Let $\sigma \in \mathbb{R}$ and $\epsilon>0$. For $t \in \mathbb{R}$ with $|t|\geq 1$, we have
    \begin{equation}
        \lv \zeta(\sigma+it)\rv \ll |t|^{\alpha+\epsilon},
    \end{equation}
    where $\alpha=0$ for $\sigma>1$, $\alpha = \frac{1}{2}(1-\sigma)$ for $\sigma \in [0,1]$, and $\alpha=\frac{1}{2}-\sigma$ for $\sigma<0$. The implied constant depends on $\sigma$ and $\epsilon$.
\end{lem}

\begin{lem} \label{extend}
    Let $\Delta_{nm} =  \frac{T}{2\pi}-\frac{nmq}{p}$ and
    \begin{equation}
     \mathcal{F}=2\pi \lc \frac{q}{p} \rc^{\h} \sum_{\substack{n,m\geq 1 \\ |\Delta_{nm}| > H(pqH)^{\epsilon}}}\exp\lc -\frac{(2\pi nmq - Tp)^{2}}{(Hp)^2} \rc e\lc -\frac{nmq}{p} \rc.
     \end{equation}
     Then we have
     \begin{equation}
         \mathcal{F}= O\lc e^{-(pqH)^{\epsilon}} \rc
     \end{equation}
     \begin{proof}
        It is enough to show
        \begin{equation}
            \sum_{\substack{n,m\geq 1 \\ |\Delta_{nm}| > H(pqH)^{\epsilon}}}\exp\lc -\frac{(2\pi nmq - Tp)^{2}}{(Hp)^2} \rc \ll e^{-(pqH)^{\epsilon}}.
        \end{equation}
        We first split the sum at $nm=T^{100}$. For the case $nm > T^{100}$, we have
        \begin{equation}
        \begin{aligned}
            \sum_{nm > T^{100}}\exp\lc -\frac{(2\pi nmq - Tp)^{2}}{(Hp)^2} \rc &= \sum_{r > T^{100}} d(r) \exp\lc -\frac{(2\pi rq - Tp)^{2}}{(Hp)^2} \rc\\
            &\ll \int_{T^{100}}^{\infty} t \exp\lc -\lc\frac{2\pi q}{Hp}\rc^2\lc t-\frac{Tp}{2\pi q}\rc^2 \rc dt\\
            &\ll e^{-T^{100}}.
        \end{aligned}
        \end{equation}
        For the case $nm \leq T^{100}$, we apply $ -\Delta_{nm}^2 < -H^2 (pqH)^{2\epsilon}$ and get
        \begin{equation}
        \begin{aligned}
            \sum_{\substack{nm\leq T^{100} \\ |\Delta_{nm}| > H(pqH)^{\epsilon}}}\exp\lc -\frac{(2\pi nmq - Tp)^{2}}{(Hp)^2} \rc &\ll \sum_{\substack{nm\leq T^{100} \\ |\Delta_{nm}| > H(pqH)^{\epsilon}}}\exp\lc -(pqH)^{2\epsilon} \rc\\
            &\ll \exp\lc -(pqH)^{2\epsilon} \rc \sum_{r\leq T^{100}} d(r)\\
            &\ll T^{100+\epsilon}e^{-(pqH)^{2\epsilon}}.
        \end{aligned}
        \end{equation}
        Finally, we have
        \begin{equation}
            T^{100+\epsilon}e^{-(pqH)^{2\epsilon}} \ll e^{-(pqH)^{\epsilon}}.
        \end{equation}
     \end{proof}
\end{lem}

\begin{lem} \label{residue}
    Let
    \begin{equation}
        B(w) = \frac{2}{q^{w}}-\frac{1}{(pq)^{w}}-\frac{1}{p-1}\lc \frac{p}{q} \rc^{w} \lc 1- \frac{1}{p^{w}} \rc^2.
    \end{equation}
    Then
    \begin{equation}
        \operatorname{Res}_{w=1} B(w)\GT(w)\zeta^2(w) = \frac{2 \pi i}{q}\lc G_{T,H}'(1)-G_{T,H}(1)\log(pq)+2\gamma G_{T,H}(1) \rc.
    \end{equation}
    \begin{proof}
        Note that $B(w)$ and $\GT(w)$ are analytic at $w=1$, and $\zeta^2(w)$ has a double pole at $w=1$. Thus, we have
        \begin{equation}\label{rc}
        \begin{aligned}
            \operatorname{Res}_{w=1}& B(w)\GT(w)\zeta^2(w) \\
            &= 2\pi i\lim_{w \xrightarrow{}1} \frac{d}{dw} (w-1)^2B(w)\GT(w)\zeta^2(w)\\
            &=2\pi i \lc B(1)\GT(1)\lim_{w \xrightarrow{}1}\frac{d}{dw} (w-1)^2\zeta^2(w) + \lim_{w \xrightarrow{}1}(w-1)^2\zeta^2(w)\frac{d}{dw}B(w)\GT(w)\rc.
        \end{aligned}
        \end{equation}
        Observe that
        \begin{equation}\label{r1}
            B(1)=\frac{1}{q},
        \end{equation}
        \begin{equation}
            B'(w) = -B(w)\log q +\frac{\log p}{q^w} \lc \frac{p-p^{2w}}{p^w(p-1)} \rc,
        \end{equation}
        and
        \begin{equation}\label{r2}
            B'(1) = -\frac{\log(pq)}{q}.
        \end{equation}
        Moreover, the Laurent series of $\zeta^2(w)$ at $w=1$ is
        \begin{equation}
            \zeta^2(w) = \frac{1}{(w-1)^2}+\frac{2\gamma}{w-1} + O\lc 1 \rc.
        \end{equation}
        Thus, we have
        \begin{equation}\label{r3}
            \lim_{w \xrightarrow{}1}(w-1)^2\zeta^2(w) = 1
        \end{equation}
        and
        \begin{equation}\label{r4}
            \lim_{w \xrightarrow{}1}\frac{d}{dw} (w-1)^2\zeta^2(w) = 2\gamma.
        \end{equation}
        Finally, substituting \eqref{r1}, \eqref{r2}, \eqref{r3}, and \eqref{r4} into \eqref{rc}, we have
        \begin{equation}
        \operatorname{Res}_{w=1} B(w)\GT(w)\zeta^2(w) = \frac{2 \pi i}{q}\lc G_{T,H}'(1)-G_{T,H}(1)\log(pq)+2\gamma G_{T,H}(1) \rc.
        \end{equation}
    \end{proof}
\end{lem}

\begin{lem} \label{epsilonbound}
    Let $\epsilon \in (0,1)$ and
    \begin{equation}
        B(\epsilon+it) = \frac{2}{q^{\epsilon+it}}-\frac{1}{(pq)^{\epsilon+it}}-\frac{1}{p-1}\lc \frac{p}{q} \rc^{\epsilon+it} \lc 1- \frac{1}{p^{\epsilon+it}} \rc^2.
    \end{equation}
    Then we have
    \begin{equation}
        \lc \frac{q}{p} \rc^{\h} \inftyInt B(\epsilon+it) G_{T,H}(\epsilon+it)\zeta^2(\epsilon+it)dt = O\lc \frac{T}{H}(pqH)^{\epsilon}\lc \frac{q}{p} \rc^{\h} \rc.
    \end{equation}
    \begin{proof}
        We need to show
        \begin{equation}
            \lc \frac{q}{p} \rc^{\h} \inftyInt \lv B(\epsilon+it) G_{T,H}(\epsilon+it)\zeta^2(\epsilon+it) \rv dt \ll \frac{T}{H}(pqH)^{\epsilon}\lc \frac{q}{p} \rc^{\h}.
        \end{equation}
        Observe that we have
        \begin{equation}
            \lv \frac{2}{q^{\epsilon+it}}-\frac{1}{(pq)^{\epsilon+it}}-\frac{1}{p-1}\lc \frac{p}{q} \rc^{\epsilon+it} \lc 1- \frac{1}{p^{\epsilon+it}} \rc^2 \rv \ll (pq)^{\epsilon}
        \end{equation}
        and, by Lemma \ref{conv}, 
        \begin{equation}
            \lv \zeta^2(\epsilon+it) \rv \ll \lv t \rv^{1+\epsilon}.
        \end{equation}
        Then we have
        \begin{equation}
            \lc \frac{q}{p} \rc^{\h} \inftyInt \lv B(\epsilon+it) G_{T,H}(\epsilon+it)\zeta^2(\epsilon+it) \rv dt \ll \lc \frac{q}{p} \rc^{\h}(pq)^{\epsilon} \inftyInt \lv G_{T,H}(\epsilon+it) \rv \lv t \rv^{1+\epsilon} dt.
        \end{equation}
        Now it remains to show
        \begin{equation}
            \inftyInt \lv G_{T,H}(\epsilon+it) \rv \lv t \rv^{1+\epsilon} dt \ll \frac{T}{H}T^{\epsilon}.
        \end{equation}
        Split the integral at $|t|=\frac{T}{H}$. For $|t|\leq \frac{T}{H}$, we apply Lemma \ref{mellin2} and get
        \begin{equation}
            \int_{|t|\leq \frac{T}{H}} \lv G_{T,H}(\epsilon+it) \rv \lv t \rv^{1+\epsilon} dt \ll \int_{|t|\leq \frac{T}{H}} \GT(\epsilon) \lv t \rv^{1+\epsilon} dt \ll HT^{\epsilon-1} \int_{|t|\leq \frac{T}{H}} \lv t \rv^{1+\epsilon} dt.
        \end{equation}
        Since $\int_{|t|\leq \frac{T}{H}} \lv t \rv^{1+\epsilon} dt = 2\int_{0}^{\frac{T}{H}} t^{1+\epsilon} dt$, we have
        \begin{equation}
            HT^{\epsilon-1} \int_{|t|\leq \frac{T}{H}} \lv t \rv^{1+\epsilon} dt \ll HT^{\epsilon-1} \int_{0}^{\frac{T}{H}} t^{1+\epsilon} dt \ll HT^{\epsilon-1} \lc \frac{T}{H} \rc^{2+\epsilon} \ll \frac{T}{H}T^{\epsilon}.
        \end{equation}
        For $|t|> \frac{T}{H}$, we apply Lemma \ref{weight} and get
        \begin{equation} \label{l9eq}
            \int_{|t| > \frac{T}{H}} \lv G_{T,H}(\epsilon+it) \rv \lv t \rv^{1+\epsilon} dt \ll_n \int_{|t| > \frac{T}{H}} HT^{\epsilon-1}\frac{\lc \frac{T}{H} \rc^n}{(1+|t|)^n}\lv t \rv^{1+\epsilon} dt \ll HT^{\epsilon-1}\lc \frac{T}{H} \rc^n \int_{|t| > \frac{T}{H}} \lv t \rv^{1+\epsilon-n} dt.
        \end{equation}
        For $n\geq 3$, $\int_{|t| > \frac{T}{H}} \lv t \rv^{1+\epsilon-n} dt$ converges. Thus, we get
        \begin{equation}
            \int_{|t| > \frac{T}{H}} \lv t \rv^{1+\epsilon-n} dt \ll \int_{\frac{T}{H}}^{\infty} \lv t \rv^{1+\epsilon-n} dt \ll \lc \frac{T}{H} \rc^{2+\epsilon-n}.
        \end{equation}
        Substituting the result into \eqref{l9eq}, we obtain
        \begin{equation}
            HT^{\epsilon-1}\lc \frac{T}{H} \rc^n \int_{|t| > \frac{T}{H}} \lv t \rv^{1+\epsilon-n} dt \ll HT^{\epsilon-1}\lc \frac{T}{H} \rc^n \lc \frac{T}{H} \rc^{2+\epsilon-n}\ll \frac{T}{H}T^{\epsilon}.
        \end{equation}
    \end{proof}
\end{lem}

\section{Proof of Theorem \ref{thm1}}
Define
\begin{equation}
    F(s,p,q)=\inftyInt \lc\frac{p}{q}\rc^{it} \zeta\lc\tfrac{1}{2}-s-it\rc \lc\zeta\lc\tfrac{1}{2}+s+it\rc-\lc s-\tfrac{1}{2}+it\rc^{-1}\rc \exp\lc-\frac{(t-T)^2}{H^2}\rc dt.
\end{equation}
This function is analytic for $\mathrm{Re}(s)>-\h$, and, for $s=0$, we have
\begin{equation} \label{break}
    F(0,p,q)=\inftyInt  \lc\frac{p}{q}\rc^{it} \lv \zeta\lc\tfrac{1}{2}+it\rc \rv^2 \exp\lc-\frac{(t-T)^2}{H^2}\rc dt - \inftyInt  \lc\frac{p}{q}\rc^{it} \frac{\zeta\lc\tfrac{1}{2}-it\rc}{-\tfrac{1}{2}+it} \exp\lc-\frac{(t-T)^2}{H^2}\rc dt.
\end{equation}
Observe that the first integral in \eqref{break} is precisely the left hand side of \eqref{eqthm1}. For the second integral, we apply Cauchy-Schwarz inequality and obtain
\begin{equation}
    \lv\inftyInt  \lc\frac{p}{q}\rc^{it} \frac{\zeta\lc\tfrac{1}{2}-it\rc}{-\tfrac{1}{2}+it} \exp\lc-\frac{(t-T)^2}{H^2}\rc dt\rv^2 \ll H\inftyInt  \lv \frac{\zeta\lc\tfrac{1}{2}-it\rc}{-\tfrac{1}{2}+it}\rv^2 \exp\lc-\frac{(t-T)^2}{H^2}\rc dt.
\end{equation}
By Lemma \ref{intbound}, we have
\begin{equation}
    H\inftyInt  \lv \frac{\zeta\lc\tfrac{1}{2}-it\rc}{-\tfrac{1}{2}+it}\rv^2 \exp\lc-\frac{(t-T)^2}{H^2}\rc dt \ll \frac{H^2}{T^2}T^{\epsilon}.
\end{equation}
Thus, the second integral in \eqref{break} is of size $O\lc \frac{H}{T}T^{\epsilon} \rc$, which is negligible. To obtain the right hand side of \eqref{eqthm1}, our plan is to express $F(s,p,q)$ in a different form and analytically continue it to a neighborhood of $s=0$.

Following the same procedure as in \cite[Section 3.2-3.3]{khan25}, we obtain
\begin{equation} \label{Fs}
    F(s,p,q) = F_1(s,p,q) + F_2(s,p,q) + F_3(s,p,q),
\end{equation}
where
\begin{align}
    F_1(s,p,q) &= -2\infS \frac{s-\h}{(2\pi nm)^{\h + s}} \int_{0}^{\infty}x^{-\frac{3}{2}+s}\sin{x}\inftyInt \lc \frac{xp}{2\pi nmq} \rc^{it} \exp\lc-\frac{(t-T)^2}{H^2}\rc dt dx, \\
    F_2(s,p,q) &= -2\infS \frac{1}{(2\pi nm)^{\h + s}} \int_{0}^{\infty}x^{-\frac{3}{2}+s}\sin{x}\inftyInt it\lc \frac{xp}{2\pi nmq} \rc^{it} \exp\lc-\frac{(t-T)^2}{H^2}\rc dt dx, \\
    F_3(s,p,q) &= 2\sum_{n \geq 1} \frac{1}{(2\pi n)^{\h + s}} \int_{0}^{\infty}x^{-\frac{3}{2}+s}\sin{x}\inftyInt \lc \frac{xp}{2\pi nq} \rc^{it} \exp\lc-\frac{(t-T)^2}{H^2}\rc dt dx,
\end{align}
and these expressions are defined on Re$(s)\in \lc \h,\frac{3}{2} \rc$. We now analyze $F_1(s,p,q), F_2(s,p,q), F_3(s,p,q)$ separately and show that each term admits an analytic continuation to a neighborhood of $s=0$. Then we begin with $F_1(s,p,q)$ by evaluating the $t$ integral
\begin{equation} \label{fourier1}
    \inftyInt \exp(itL) \exp\lc -\frac{(t-T)^2}{H^2} \rc dt = H\sqrt{\pi}\exp\lc -\lc \frac{HL}{2}\rc^2\rc \exp \lc iLT\rc,
\end{equation}
where $L=\log\lc \frac{xp}{2\pi nmq} \rc$. Substituting \eqref{fourier1} into $F_1(s)$ and setting $y=\frac{HL}{2}$, we obtain
\begin{equation}
    F_1(s,p,q) = -\frac{2}{\sqrt{\pi}} \infS\frac{s-\h}{nm} \inftyInt \lc \frac{q}{p} \rc^{-\h + s} \exp \lc -y^2+\frac{2s-1}{H}y \rc e \lc \frac{T}{H\pi}y \rc \sin \lc \frac{2\pi nm q}{p}e^{\frac{2y}{H}} \rc dy.
\end{equation}
By Lemma \ref{bound}, this expression admits an analytic continuation to
a neighborhood of $s=0$. Moreover, after restricting the sum to $n,m \leq T^{50}$, the contribution of the omitted terms is of size $O\lc \lc\frac{p}{q}\rc^{\frac{3}{2}} T^{-47}\rc$. Therefore, we have
\begin{align*}
    \lv F_1(0,p,q) \rv &\ll \lv\frac{1}{\sqrt{\pi}} \sum_{n,m \leq T^{50}}\frac{1}{nm} \inftyInt \lc \frac{q}{p} \rc^{-\h} \exp \lc -y^2-\frac{y}{H} \rc e \lc \frac{T}{H\pi}y \rc \sin \lc \frac{2\pi nm q}{p}e^{\frac{2y}{H}} \rc dy \rv\\
    &\ll \lc \frac{p}{q} \rc^{\h}\sum_{n,m \leq T^{50}}\frac{1}{nm}\\
    &\ll \lc \frac{p}{q} \rc^{\h}T^{\epsilon}.
\end{align*}
Similarly, after applying the same argument to $F_3(s)$, we obtain
\[|F_3(0,p,q)|\ll \lc \frac{p}{q} \rc^{\h}T^{\epsilon}\]
Since both $F_1(0,p,q)$ and $F_3(0,p,q)$ are of size of $O\lc \lc\frac{p}{q}\rc^{\h} T^{\epsilon}\rc$, the main contribution should come from $F_2(0,p,q)$. We next evaluate the $t$ integral
\begin{equation} \label{fourier2}
    \inftyInt it\exp(itL) \exp\lc-\frac{(t-T)^2}{H^2}\rc dt = -\frac{H\sqrt{\pi}}{2}\lc H^2L-2iT \rc \exp\lc -\lc\frac{HL}{2}\rc^2+iTL \rc,
\end{equation}
where $L=\log\lc \frac{xp}{2\pi nmq} \rc$. Substituting \eqref{fourier2} into $F_2(s,p,q)$ and set $y=\frac{HL}{2}$, we have
\begin{equation} \label{F2}
\begin{aligned}
    F_2&(s,p,q)=\\
    &\frac{2}{\sqrt{\pi}} \infS\frac{1}{nm} \inftyInt \lc \frac{q}{p} \rc^{-\h + s} \exp \lc -y^2+\frac{2s-1}{H}y \rc e \lc \frac{T}{H\pi}y \rc \sin \lc \frac{2\pi nm q}{p}e^{\frac{2y}{H}} \rc \lc Hy-iT \rc dy.
\end{aligned}
\end{equation}
Applying Lemma \ref{bound}, we can analytically continue $F_2(s,p,q)$ to a neighborhood of $s=0$ and restrict the sum to $n,m \leq T^{50}$ with an error of size $ O\lc \lc\frac{p}{q}\rc^{\frac{3}{2}} T^{-47}\rc$. Expressing the sine function in exponential form, we have 
\begin{equation} \label{F2}
    F_2(0,p,q)=\frac{1}{i\sqrt{\pi}}\lc  \mathcal{H}_{1,+} +  \mathcal{H}_{1,-} +  \mathcal{H}_{2,+} +  \mathcal{H}_{2,-}\rc  + O\lc \lc\frac{p}{q}\rc^{\frac{3}{2}} T^{-47}\rc,
\end{equation}
where
\begin{align}
    \mathcal{H}_{1,\pm}&:= \pm \fST \lc \frac{p}{q} \rc^{\h} \frac{1}{nm} \inftyInt Hy \exp \lc -y^2-\frac{y}{H} \rc e \lc \frac{Ty}{H\pi} \pm \frac{nmq}{p}e^{\frac{2y}{H}} \rc dy, \label{eq:H1}\\
    \mathcal{H}_{2,\pm}&:= \mp \fST \lc \frac{p}{q} \rc^{\h} \frac{i}{nm} \inftyInt T \exp \lc -y^2-\frac{y}{H} \rc e \lc \frac{Ty}{H\pi} \pm \frac{nmq}{p}e^{\frac{2y}{H}} \rc dy .\label{eq:H2}
\end{align}
Furthermore, by Lemma \ref{bound2}, the sums in \eqref{eq:H1} and \eqref{eq:H2} can be restricted to the range $|\frac{T}{2\pi} \pm \frac{nmq}{p}| \leq H(pqH)^{\epsilon}$ up to the errors of size $O\lc (pqT)^{\epsilon} \lc \frac{p}{q} \rc^{\h} \rc$ and $O\lc \frac{T}{H}(pqT)^{\epsilon} \lc \frac{p}{q} \rc^{\h} \rc$ respectively. Since $\mathcal{H}_{1,+}$ and $\mathcal{H}_{2,+}$ can be absorbed into the error by Lemma \ref{bound2}, it remains to analyze only $\mathcal{H}_{1,-}$ and $\mathcal{H}_{2,-}$. Applying Lemma \ref{a1} on $\mathcal{H}_{1,-}$ and $\mathcal{H}_{2,-}$, we have
\begin{equation}
\begin{aligned}
    \mathcal{H}_{1,-} =& \sum_{\lv \frac{nmq}{p} - \frac{T}{2\pi} \rv \leq H(pqH)^{\epsilon}}\frac{-1}{nm} \inftyInt \lc \frac{p}{q} \rc^{\h} Hy \exp \lc -y^2 \rc e \lc \frac{Ty}{H\pi} - \frac{2nmqy}{Hp} \rc e\lc -\frac{nmq}{p} \rc dy \\
    &+ O\lc (pqT)^{\epsilon} \lc \frac{p}{q} \rc^{\h} \rc,
\end{aligned}
\end{equation}
\begin{equation}
\begin{aligned}
    \mathcal{H}_{2,-} =& \sum_{\lv \frac{nmq}{p} - \frac{T}{2\pi} \rv \leq H(pqH)^{\epsilon}}\frac{i}{nm} \inftyInt \lc \frac{p}{q} \rc^{\h} T \exp \lc -y^2 \rc e \lc \frac{Ty}{H\pi} - \frac{2nmqy}{Hp} \rc e\lc -\frac{nmq}{p} \rc dy \\
    &+ O\lc \frac{T}{H} \lc pqT \rc^{\epsilon} \lc \frac{p}{q} \rc^{\h} \rc.
\end{aligned}
\end{equation}
We now evaluate the integrals
\begin{equation} \label{i1}
    \inftyInt y\exp(-y^2)\exp(iLy)dy = \exp\lc -\frac{L^2}{4} \rc \frac{iL\sqrt{\pi}}{2},
\end{equation}
\begin{equation} \label{i2}
    \inftyInt \exp(-y^2)\exp(iLy)dy = \exp\lc -\frac{L^2}{4} \rc \sqrt{\pi},
\end{equation}
where $L=\frac{2Tp-4\pi nmq}{Hp}$.
\begin{myremark} \label{remark1}
     In analogy with the results of \cite{Atk49} and \cite{Jut83}, one may expect the restriction $\delta>\h$ in Theorem \ref{thm1} to be improvable to $\delta > \frac{1}{3}$. Indeed, by retaining the quadratic term after the Taylor expansion in equation \eqref{taylor3} of Lemma \ref{a1}, this improvement may be expected. However, this introduces an additional quadratic term into the oscillatory phase in \eqref{i1} and \eqref{i2}, making the resulting integrals and the following steps more complicated to evaluate. We therefore do not pursue this refinement in the current work, but hope to revisit it in the future.
\end{myremark}
Substituting \eqref{i1} and \eqref{i2} into $\mathcal{H}_{1,-}$ and $\mathcal{H}_{2,-}$ respectively, we have
\begin{equation} \label{H1}
\begin{aligned}
    \mathcal{H}_{1,-} =& \sum_{\lv \frac{nmq}{p} - \frac{T}{2\pi} \rv \leq H(pqH)^{\epsilon}}\frac{i\sqrt{\pi}}{nm}\lc \frac{p}{q} \rc^{\h}\lc \frac{2\pi nmq - Tp}{p}\rc \exp\lc -\frac{(2\pi nmq - Tp)^{2}}{(Hp)^2} \rc e\lc -\frac{nmq}{p} \rc  \\
    &+ O\lc (pqT)^{\epsilon} \lc \frac{p}{q} \rc^{\h} \rc,\\
\end{aligned}
\end{equation}
\begin{equation} \label{H2}
\begin{aligned}
    \mathcal{H}_{2,-} =& \sum_{\lv \frac{nmq}{p} - \frac{T}{2\pi} \rv \leq H(pqH)^{\epsilon}}\frac{Ti\sqrt{\pi}}{nm}\lc \frac{p}{q} \rc^{\h}\exp\lc -\frac{(2\pi nmq - Tp)^{2}}{(Hp)^2} \rc e\lc -\frac{nmq}{p} \rc \\
    &+ O\lc \frac{T}{H} \lc pqT \rc^{\epsilon} \lc \frac{p}{q} \rc^{\h} \rc.
\end{aligned}
\end{equation}
Substituting \eqref{H1} and \eqref{H2} into the expression \eqref{F2} and extending the
restricted sum back to an infinite sum with the cost of a negligible error by Lemma \ref{extend}, we have
\begin{equation} \label{F2+}
    F_2(0,p,q) = 2\pi \lc \frac{q}{p} \rc^{\h} \infS\exp\lc -\frac{(2\pi nmq - Tp)^{2}}{(Hp)^2} \rc e\lc -\frac{nmq}{p} \rc + O\lc \frac{T}{H} \lc pqT \rc^{\epsilon} \lc \frac{p}{q} \rc^{\h} \rc. 
\end{equation}
Substituting \eqref{F2+} into $F(0,p,q)$ and splitting the sum according to the residue class of $nm \bmod p$, we obtain
\begin{equation} \label{F0}
    F(0,p,q) = 2\pi \lc \frac{q}{p} \rc^{\h} \left[ S_1 + S_2 \right] + O\lc \frac{T}{H} \lc pqT \rc^{\epsilon} \lc \frac{p}{q} \rc^{\h} \rc,
\end{equation}
where
\begin{equation}
    S_1 = \sum_{\substack{n,m \geq 1 \\ nm \equiv 0 \bmod p}} \exp\lc -\lc\frac{2\pi nm q - Tp}{Hp}\rc^2 \rc
\end{equation}
\begin{equation}
    S_2 = \sum_{\substack{r \bmod p\\ r \neq 0}} \sum_{\substack{n,m \geq 1 \\ nm \equiv r \bmod p}} e\lc -\frac{rq}{p} \rc \exp\lc -\lc\frac{2\pi nm q - Tp}{Hp}\rc^2 \rc
\end{equation}
We first consider $S_1$. Applying the inclusion-exclusion principle, we can rewrite $S_1$ as the following expression
\begin{equation}
    S_1 = 2\infS \exp\lc -\lc\frac{2\pi nm q - T}{H}\rc^2 \rc - \infS\exp\lc -\lc\frac{2\pi nm pq - T}{H}\rc^2 \rc.
\end{equation}
Applying Lemma \ref{mellininversion} to the exponential term and interchanging the sums and integrals, we have
\begin{equation}
    S_1 = \frac{1}{2 \pi i} \int_{(c)} \lc \frac{2}{q^w}-\frac{1}{(pq)^w} \rc G_{T,H}(w)\zeta^2(w) dw,
\end{equation}
where $c \in (1,\infty)$. We next consider $S_2$. By the orthogonality of Dirichlet character modulo $p$, the congruence condition $nm \equiv r \bmod p$ can be rewritten as
\begin{equation}
    1_{nm \equiv r \bmod p}=\frac{1}{p-1}\sum_{\chi\bmod p}\chi(nm)\overline{\chi}(r).
\end{equation}
Then we have
\begin{align}
    S_2 &= \sum_{\substack{r \bmod p\\ r \neq 0}} \sum_{n,m \geq 1} \frac{1}{p-1}\sum_{\chi\bmod p}\chi(nm)\overline{\chi}(r)  e\lc -\frac{rq}{p} \rc \exp\lc -\lc\frac{2\pi nm q - Tp}{Hp}\rc^2 \rc\\
    &=\frac{1}{p-1}\sum_{\substack{r \bmod p\\ r \neq 0}} e\lc -\frac{rq}{p} \rc \sum_{\chi \bmod p} \overline{\chi}(r) \infS \chi(nm) \exp\lc -\lc\frac{2\pi nm q - Tp}{Hp}\rc^2 \rc.
\end{align}
Replacing $r$ by $r\overline{q}$ and applying Lemma \ref{mellininversion} to the exponential term, we obtain
\begin{equation}
    S_2 = \frac{1}{p-1} \sum_{\chi \bmod p} \chi(q)  \sum_{\substack{r \bmod p\\ r \neq 0}} e\lc -\frac{r}{p} \rc \overline{\chi}(r) \infS \frac{\chi(nm)}{2\pi i} \int_{(c)} \lc \frac{nmq}{p} \rc^{-w} G_{T,H}(w) dw
\end{equation}
for $c \in (1,\infty)$. Isolating the principal character and interchanging the sum and the integral, we have
\begin{equation}
\begin{aligned}
    S_2 =& \frac{1}{2\pi i(p-1)} \sumstar_{\chi \bmod p} \chi(q)  \sum_{\substack{r \bmod p\\ r \neq 0}} e\lc -\frac{r}{p} \rc \overline{\chi}(r) \int_{(c)} \lc \frac{p}{q} \rc^{w} G_{T,H}(w) L^2(w,\chi) dw\\
    &+ \frac{1}{2\pi i(p-1)} \chi_{0}(q)  \sum_{\substack{r \bmod p\\ r \neq 0}} e\lc -\frac{r}{p} \rc \overline{\chi_{0}}(r) \int_{(c)} \lc \frac{p}{q} \rc^{w} G_{T,H}(w) \lc 1- \frac{1}{p^w} \rc^2 \zeta^2(w) dw.
\end{aligned}
\end{equation}
Since we have $\chi_0(q) = 1$, $\sum\limits_{\substack{r \bmod p\\ r \neq 0}} e\lc -\frac{r}{p} \rc = -1$, and $\sum\limits_{\substack{r \bmod p\\ r \neq 0}} e\lc -\frac{r}{p} \rc \overline{\chi}(r) = \chi(-1) \tau(\overline{\chi})$, we obtain
\begin{equation}
\begin{aligned}
    S_2 =& \frac{1}{2\pi i(p-1)} \sumstar_{\chi \bmod p} \chi(-q) \tau(\overline{\chi})    \int_{(c)} \lc \frac{p}{q} \rc^{w} G_{T,H}(w) L^2(w,\chi) dw\\
    &- \frac{1}{2\pi i(p-1)} \int_{(c)} \lc \frac{p}{q} \rc^{w} G_{T,H}(w) \lc 1- \frac{1}{p^w} \rc^2 \zeta^2(w) dw.
\end{aligned}
\end{equation}
Substituting the new expressions of $S_1$ and $S_2$ into \eqref{F0}, we obtain
\begin{equation}
\begin{aligned}
    F(0,p,q) =& \lc \frac{q}{p} \rc^{\h} \frac{1}{i} \int_{(c)} \left[ \frac{2}{q^w}-\frac{1}{(pq)^w}-\frac{1}{p-1}\lc \frac{p}{q} \rc^{w} \lc 1- \frac{1}{p^w} \rc^2 \right] G_{T,H}(w)\zeta^2(w)dw\\
    &+ \lc \frac{q}{p} \rc^{\h}\frac{1}{i(p-1)} \sumstar_{\chi \bmod p} \chi(-q) \tau(\overline{\chi})    \int_{(c)} \lc \frac{p}{q} \rc^{w} G_{T,H}(w) L^2(w,\chi) dw \\
    &+ O\lc \frac{T}{H}(pqT)^{\epsilon}\lc \frac{p}{q} \rc^{\h} \rc.
\end{aligned}
\end{equation}
Shifting the first integral of $\zeta(w)$ to $w = \epsilon + it$, $\epsilon\in(0,1)$, we pick up the pole at $w=1$ of $\zeta(w)$ by Lemma \ref{residue} and obtain
\begin{equation}
\begin{aligned}
    \operatorname{Res}_{w=1}  \left[ \frac{2}{q^w}-\frac{1}{(pq)^w}-\frac{1}{p-1}\lc \frac{p}{q} \rc^{w} \lc 1- \frac{1}{p^w} \rc^2 \right] &G_{T,H}(w)\zeta^2(w) \\
    &= \frac{2\pi i}{q}\lc \GT'(1) - \log(pq) \GT(1) + 2\gamma \GT(1) \rc,
\end{aligned}
\end{equation}
which contributes to the main term of Theorem \ref{thm1}. For the shifted integral, we obtain
\begin{equation}
    \lc \frac{q}{p} \rc^{\h} \frac{1}{i} \int_{(\epsilon)} \left[ \frac{2}{q^w}-\frac{1}{(pq)^w}-\frac{1}{p-1}\lc \frac{p}{q} \rc^{w} \lc 1- \frac{1}{p^w} \rc^2 \right] G_{T,H}(w)\zeta^2(w)dw = O\lc \frac{T}{H}(pqT)^{\epsilon}\lc \frac{q}{p} \rc^{\h} \rc
\end{equation}
by applying Lemma \ref{epsilonbound}. Eventually, we have
\begin{equation}
\begin{aligned}
    F(0,p,q) =& \frac{2\pi}{\sqrt{pq}}\lc G_{T,H}'(1) - G_{T,H}(1)\log(pq) + 2\gamma G_{T,H}(1)\rc \\
    &+ \lc\frac{q}{p}\rc^{\h}\frac{1}{i(p-1)} \sumstar_{\chi \bmod p} \chi(-q) \tau(\overline{\chi})    \int_{(c)} \lc \frac{p}{q} \rc^{w} G_{T,H}(w) L^2(w,\chi) dw\\
    &+O\lc \frac{T}{H}(pqT)^{\epsilon} \left[\lc \frac{p}{q} \rc^{\h} + \lc \frac{q}{p} \rc^{\h} \right] \rc.
\end{aligned}
\end{equation}
Finally, we shift the remaining integral to $c=\h$, substitute $w=\h+it$, apply the functional equation on one of the factors $L(w,\chi)$ using Lemma \ref{functional}, and obtain the expression in Theorem \ref{thm1}.

\section{Proof of Corollary \ref{cor1}}
Observe that
\begin{equation}
    \overline{\inftyInt  \lc\frac{p}{q}\rc^{it} \lv \zeta\lc\tfrac{1}{2}+it\rc \rv^2 \exp\lc-\frac{(t-T)^2}{H^2}\rc dt} = \inftyInt  \lc\frac{p}{q}\rc^{it} \lv \zeta\lc\tfrac{1}{2}+it\rc \rv^2 \exp\lc-\frac{(t+T)^2}{H^2}\rc dt,
\end{equation}
and following the same proof as last section, we obtain a conjugate version of Theorem \ref{thm1}.
\begin{cor}\label{cor2}
    Let $p,q$ be distinct odd primes, $T>1$ and $H=T^{\delta}$ for $\delta \in \lc \h,1 \rc$. We have
    \begin{equation}\label{eqcor2}
    \begin{aligned}
        \inftyInt \lc \frac{p}{q} \rc^{it} &\lv \zeta\lc \tfrac{1}{2}+it \rc \rv^2 \exp\lc -\frac{(t+T)^2}{H^2} \rc dt =\\
        &\frac{2 \pi}{\sqrt{pq}}\lc G_{T,H}'(1)-G_{T,H}(1)\log(pq)+2\gamma G_{T,H}(1) \rc \\
        &+ \sumstar_{\chi \bmod p} \frac{\sqrt{p}}{i^{\alpha}(p-1)}\chi(-q) \inftyInt G_{T,H}\lc \tfrac{1}{2}+it \rc \lc \frac{\pi}{q}\rc^{it} \frac{\Gamma \lc \frac{1+2\alpha-2it}{4} \rc}{\Gamma \lc \frac{1+2\alpha+2it}{4} \rc} \lv L\lc \tfrac{1}{2}+it,\chi \rc \rv^2 dt\\
        &+ O\lc \frac{T}{H}(pqT)^{\epsilon} \left[\lc \frac{p}{q} \rc^{\h} + \lc \frac{q}{p} \rc^{\h} \right] \rc,
    \end{aligned}
    \end{equation}
    for $\alpha= 0$ if $\chi(-1)=1$ and $\alpha=1$ if $\chi(-1)=-1$.
\end{cor}
Then, we define the following
\begin{equation}
    J_1(p,q) = \inftyInt  \lc\frac{p}{q}\rc^{it} \lv \zeta\lc\tfrac{1}{2}+it\rc \rv^2 \lc \exp\lc-\frac{(t-T)^2}{H^2}\rc + \exp\lc-\frac{(t+T)^2}{H^2}\rc \rc dt,
\end{equation}
\begin{equation}
    J_2(p,q) = \inftyInt  \lc\frac{p}{q}\rc^{it} \lv \zeta\lc\tfrac{1}{2}+it\rc \rv^2 \lc \exp\lc-\frac{(t-T)^2}{H^2}\rc - \exp\lc-\frac{(t+T)^2}{H^2}\rc \rc dt.
\end{equation}
Observe that
\begin{equation}
    \overline{\inftyInt  \lc\frac{p}{q}\rc^{it} \lv \zeta\lc\tfrac{1}{2}+it\rc \rv^2 \exp\lc-\frac{(t\pm T)^2}{H^2}\rc dt} = \inftyInt  \lc\frac{q}{p}\rc^{it} \lv \zeta\lc\tfrac{1}{2}+it\rc \rv^2 \exp\lc-\frac{(t\pm T)^2}{H^2}\rc dt
\end{equation}
Then we get
\begin{equation} \label{conj1}
    J_1(q,p) = \overline{J_1(p,q)} = J_1(p,q),
\end{equation}
\begin{equation} \label{conj2}
    J_2(q,p) = \overline{J_2(p,q)} = -J_2(p,q).
\end{equation}
Moreover, by Corollary \ref{cor2}, we have
\begin{equation}
\begin{aligned}
    J_1(p,q) &=\frac{4 \pi}{\sqrt{pq}}\lc G_{T,H}'(1)-G_{T,H}(1)\log(pq)+2\gamma G_{T,H}(1) \rc \\
        &+ 2\sumstar_{\substack{\chi \bmod p\\ \chi(-1) = 1}} \frac{\sqrt{p}}{p-1}\chi(q) \inftyInt G_{T,H}\lc \tfrac{1}{2}+it \rc \lc \frac{\pi}{q}\rc^{it} \frac{\Gamma \lc \frac{1-2it}{4} \rc}{\Gamma \lc \frac{1+2it}{4} \rc} \lv L\lc \tfrac{1}{2}+it,\chi \rc \rv^2 dt\\
        &+ O\lc \frac{T}{H}(pqT)^{\epsilon} \left[\lc \frac{p}{q} \rc^{\h} + \lc \frac{q}{p} \rc^{\h} \right] \rc,
\end{aligned}
\end{equation}
\begin{equation}
\begin{aligned}
    J_2(p,q) &=2\sumstar_{\substack{\chi \bmod p\\ \chi(-1) = -1}} \frac{\sqrt{p}}{i(p-1)}\chi(q) \inftyInt G_{T,H}\lc \tfrac{1}{2}+it \rc \lc \frac{\pi}{q}\rc^{it} \frac{\Gamma \lc \frac{3-2it}{4} \rc}{\Gamma \lc \frac{3+2it}{4} \rc} \lv L\lc \tfrac{1}{2}+it,\chi \rc \rv^2 dt\\
    &+ O\lc \frac{T}{H}(pqT)^{\epsilon} \left[\lc \frac{p}{q} \rc^{\h} + \lc \frac{q}{p} \rc^{\h} \right] \rc.
\end{aligned}
\end{equation}
Finally, by \eqref{conj1} and \eqref{conj2}, we obtain
\begin{equation} \label{fn1}
\begin{aligned}
    \sumstar_{\substack{\chi \bmod p\\ \chi(-1) = 1}} & \frac{\sqrt{p}}{p-1}\chi(q) \inftyInt G_{T,H}\lc \tfrac{1}{2}+it \rc \lc \frac{\pi}{q}\rc^{it} \frac{\Gamma \lc \frac{1-2it}{4} \rc}{\Gamma \lc \frac{1+2it}{4} \rc} \lv L\lc \tfrac{1}{2}+it,\chi \rc \rv^2 dt\\
    =&\sumstar_{\substack{\chi \bmod q\\ \chi(-1) = 1}} \frac{\sqrt{q}}{q-1}\chi(p) \inftyInt G_{T,H}\lc \tfrac{1}{2}+it \rc \lc \frac{\pi}{p}\rc^{it} \frac{\Gamma \lc \frac{1-2it}{4} \rc}{\Gamma \lc \frac{1+2it}{4} \rc} \lv L\lc \tfrac{1}{2}+it,\chi \rc \rv^2 dt\\ 
    &+ O\lc \frac{T}{H}(pqT)^{\epsilon} \left[\lc \frac{p}{q} \rc^{\h} + \lc \frac{q}{p} \rc^{\h} \right] \rc
\end{aligned}
\end{equation}
and
\begin{equation} \label{fn2}
\begin{aligned}
    \sumstar_{\substack{\chi \bmod p\\ \chi(-1) = -1}} & \frac{\sqrt{p}}{i(p-1)}\chi(q) \inftyInt G_{T,H}\lc \tfrac{1}{2}+it \rc \lc \frac{\pi}{q}\rc^{it} \frac{\Gamma \lc \frac{3-2it}{4} \rc}{\Gamma \lc \frac{3+2it}{4} \rc} \lv L\lc \tfrac{1}{2}+it,\chi \rc \rv^2 dt\\
    =&-\sumstar_{\substack{\chi \bmod q\\ \chi(-1) = -1}} \frac{\sqrt{q}}{i(q-1)}\chi(p) \inftyInt G_{T,H}\lc \tfrac{1}{2}+it \rc \lc \frac{\pi}{p}\rc^{it} \frac{\Gamma \lc \frac{3-2it}{4} \rc}{\Gamma \lc \frac{3+2it}{4} \rc} \lv L\lc \tfrac{1}{2}+it,\chi \rc \rv^2 dt\\ 
    &+ O\lc \frac{T}{H}(pqT)^{\epsilon} \left[\lc \frac{p}{q} \rc^{\h} + \lc \frac{q}{p} \rc^{\h} \right] \rc.
\end{aligned}
\end{equation}
Adding \eqref{fn1} and \eqref{fn2} and using $\chi(p)=\chi(-p)$ in \eqref{fn1} and $-\chi(p)=\chi(-p)$ in \eqref{fn2}, we obtain Corollary \ref{cor1}.

\printbibliography

\end{document}